\documentclass[11pt]{amsart}
\usepackage{amsmath}
\usepackage{amssymb}
\usepackage{amsfonts}
\usepackage{latexsym}
\usepackage{color}
\usepackage{graphicx}
\usepackage[left=4cm, right=4cm]{geometry}
\usepackage{soul} 
\newcommand{\be}{\begin{equation}}
\newcommand{\en}{\end{equation}}
\newcommand{\bea}{\begin{eqnarray}}
\newcommand{\ena}{\end{eqnarray}}
\newcommand{\beano}{\begin{eqnarray*}}
\newcommand{\enano}{\end{eqnarray*}}
\newcommand{\bee}{\begin{enumerate}}
\newcommand{\ene}{\end{enumerate}}

\newcommand{\mc}{\mathcal}

\newcommand{\R}{\mathbb{R}}
\newcommand{\N}{\mathbb{N}}
\newcommand{\Z}{\mathbb{Z}}

\newcommand{\B}{{\mc B}}

\newcommand{\D}{{\mc D}}

\newcommand{\cM}{{\mathcal M}}

\newtheorem{defn}{Definition}[section]
\newtheorem{thm}[defn]{Theorem}
\newtheorem{prop}[defn]{Proposition}
\newtheorem{lemma}[defn]{Lemma}
\newtheorem{cor}[defn]{Corollary}
\newtheorem{example}[defn]{Example}
\newtheorem{rem}[defn]{Remark}

\def\x{\relax\ifmmode {\mbox{*}}\else*\fi}
\newcommand{\beex}{\begin{example}$\!\!${\bf }$\;$\rm }
\newcommand{\enex}{ \end{example}}
\newcommand{\berem}{\begin{rem}$\!\!${\bf }$\;$\rm }
\newcommand{\enrem}{ \end{rem}}
\newcommand{\bedefi}{\begin{defn}$\!\!${\bf }$\;$\rm }
\newcommand{\findefi}{\end{defn}}

\newcommand{\ip}[2]{\langle {#1}|{#2}\rangle}

\def\H{{\mathcal H}}
\def\K{{\mathcal K}}
\def\J{{\mathcal J}}

\newcommand{\widechec}{\rotatebox[origin=b]{-180}{$\widehat{}$}}
\newcommand{\wch}[1]{\overset{\widechec}{#1}}

\numberwithin{equation}{section}

\begin{document}
\title[Frames and weak frames for unbounded operators]
{Frames and weak frames for unbounded operators}

\author{Giorgia Bellomonte}
\author{Rosario Corso}
\address{Dipartimento di Matematica e Informatica,
Universit\`a degli Studi di Palermo, I-90123 Palermo, Italy}
\email{giorgia.bellomonte@unipa.it}
\email{rosario.corso@studium.unict.it}

\date{\today}

\begin{abstract}In 2012 G\u{a}vru\c{t}a introduced the notions of $K$-frame and of atomic system for a linear bounded operator $K$ in a Hilbert space $\H$, in order to decompose its range $\mathcal{R}(K)$ with a frame-like expansion. In this article we revisit these concepts for an unbounded and densely defined operator $A:\D(A)\to\H$ in two different ways. In one case we consider a non-Bessel sequence where the coefficient sequence  depends continuously on $f\in\D(A)$ with respect to the norm of $\H$. In the other case we consider a Bessel sequence and the coefficient sequence depends continuously on $f\in\D(A)$ with respect to the graph norm of $A$. \end{abstract}

{\maketitle
\section{Introduction} The notion of frame in Hilbert spaces dates backs to 1952 when it was introduced in the pioneeristic paper of J. Duffin and A.C. Schaffer \cite{DuffinSchaeffer}, and was resumed in 1986 by I. Daubechies, A. Grossman and Y. Meyer in \cite{DGM}. This notion is a generalization of that of orthonormal bases. Indeed, let $\H$ be a Hilbert space with inner product $\ip{\cdot}{\cdot}$ and norm $\|\cdot\|$, a frame is a sequence in $\H$ that allows every element of $\H$ to be written as a stable, potentially infinite, linear combination of the elements of the sequence. The uniqueness of the decomposition is lost, in general, and this gives a certain freedom in the choice of the coefficients in the expansion which is in fact a good quality in applications.

 L. G\u{a}vru\c{t}a   introduced in \cite{gavruta} the notion of atomic system for a linear, {\it bounded operator} $K$ defined everywhere on $\H$. This notion generalizes frames and also  {\it atomic systems for  subspaces} in \cite{FW}. More precisely, $\{g_n\}_{n\in\mathbb{N}}\subset\H$ is an atomic system for $K$ if  there exists $\gamma >0$ such that for every $f\in \H$ there exists $a_f=\{a_n(f)\}_{n\in\mathbb{N}}\in \ell^2$, the usual Hilbert space of complex sequences, such that $\|a_f\|\leq \gamma \|f\|$ and
\begin{equation*}
Kf=\sum_{n=1}^\infty a_n(f)g_n.
\end{equation*}
 This notion turns out to be equivalent to that of $K$-frame \cite{gavruta}; i.e. a sequence $\{g_n\}_{n\in\mathbb{N}}\subset\H$ satisfying
\begin{equation}\label{K-frame}
\alpha \|K^*f\|^2\leq \sum_{n=1}^\infty |\ip{f}{g_n}|^2 \leq \beta\|f\|^2, \qquad \forall f\in \H,
\end{equation}
for some constants $\alpha,\beta>0$, where $K^*$ is the adjoint of $K$. The main theorem in \cite{gavruta} states that if $\{g_n\}_{n\in\mathbb{N}}$ is a $K$-frame, then there exists a Bessel sequence $\{h_n\}_{n\in\mathbb{N}}$ in $\H$, i.e.  $\sum_{n=1}^\infty|\ip{f}{h_n}|^2\leq\gamma\|f\|^2$ for all $f\in \H$ and some $\gamma>0$, such that

\begin{equation*}
   Kf=\sum_{n=1}^\infty \ip{f}{h_n}g_n, \qquad \forall f\in \H.
\end{equation*}

 This generalization of frames allows to write every element of $\mathcal{R} (K)$, the range of $K$, which need not be  closed, as a superposition of the elements $\{g_n\}_{n\in\mathbb{N}}$ which do not necessarily belong to $\mathcal{R} (K)$. A question can arise at this point: why develop a theory of $K$-frames since there already exists a well-studied theory of frames that reconstruct the entire space $\H$? The answer is that if in a specific situation we are looking for sequences with some properties, then we may not find any possible frame, but we may find a $K$-frame because this notion is weaker and we could want to decompose just $\mathcal{R}(K)$. 
 
 Let us see a concrete example: let $\H=L^2(\R)$, $\phi \in L^2(\R)$ and consider  the translation system $\{\phi_n(x)\}_{n\in \Z}:=\{\phi(x-cn)\}_{n\in \Z}$ and the Gabor system $G(\phi,a,b)=\{\phi_{m,n}(x)\}_{m,n\in \Z}:=\{e^{2\pi imbx}\phi(x-na)\}$ with $a,b,c>0$. As it is known \cite{ole}, there is no hope to have $\{\phi_n\}_{n\in \Z}$  (or $\{\phi_{m,n}\}_{m,n\in \Z}$ with $ab>1$) as a frame, whatever $\phi$ is in $L^2(\R)$. But if $K$ is a bounded operator on $L^2(\R)$ and $\mathcal{R}(K)\neq \H$, then we might find $\phi$ such that  one of the previous sequences is a $K$-frame.  
 
 We have taken inspiration to \cite[Example 1]{LiOgawa2} for the following simple example. We write $\widehat f$ for the Fourier transform of $f$, which is defined for $f\in L^1(\R)$ as $
 \widehat f(\gamma):=\int_\R f(x) e^{-2\pi i x \gamma} dx$, $\gamma \in \R
 $, and it is extended to $f\in L^2(\R)$ in a standard way. 
  Let  $PW_{\frac{1}{4}}=\{f\in L^2(\R): \text{supp}(\widehat{f} \!\; )\subset [-\frac{1}{4},\frac{1}{4}) \}$. 
 If $\phi \in L^2(\R)$ is such that
 	$$
 	\widehat \phi(\gamma)=
 	\begin{cases}
 	1  \hspace*{5.3cm} \text{for }|\gamma|\leq \frac{1}{4}\\
 	\text{decaying to zero continuously }  \quad\text{for }\frac{1}{4}\leq |\gamma|< \frac{1}{2}\\
 	0 \hspace*{5.3cm} \text{for }\frac{1}{2}\leq |\gamma|,
 	\end{cases}
 	$$
 	then we have for $f\in PW_{\frac{1}{4}}$
 	$$
 	\widehat f=\widehat \phi \widehat f= \widehat \phi\sum_{n\in \Z} \ip{\widehat f}{e_n}e_n=\sum_{n\in \Z} \ip{\widehat f}{e_n}\widehat \phi e_n =\sum_{n\in \Z} \ip{\widehat f}{f_n}\widehat \phi e_n  ,
 	$$
 	where
 	$$
 	e_n(\gamma)=
 	\begin{cases}
 	e^{2\pi i n \gamma}  \hspace*{0.5cm} \text{for }|\gamma|\leq \frac{1}{2}\\ 	
 	0 \hspace*{1.3cm} \text{for } |\gamma|>\frac{1}{2},
 	\end{cases} 	\mbox{and\quad } 	
 	f_n(\gamma)=
 	\begin{cases}
 		e^{2\pi i n \gamma}  \hspace*{0.5cm} \text{for }|\gamma|\leq \frac{1}{4}\\ 	
 		0 \hspace*{1.3cm} \text{for } |\gamma|>\frac{1}{4}.
 	\end{cases}  	 	$$	
 	 Thus $f=\sum_{n\in \Z} \ip{f}{\psi_n}\phi_n$ for $f\in PW_{\frac{1}{4}}$ where $\phi_n$ is the inverse Fourier transform of $\widehat \phi e_{-n}$, i.e.
 	$\phi_n(x):=\phi(x-n)$, and $\psi_n:=\wch{f_{-n}}$ is the inverse Fourier transform of $ f_{-n}$, i.e.
 	$$\psi_n(x)=\wch{f_{-n}} (x)=\begin{cases}
    4\frac{\sin\left (\frac{\pi}{2} (x-n)\right )}{\pi (x-n)} \qquad \text{if } x\neq 0\\
 	1\hspace*{2.55cm} \text{if } x= 0.
 	\end{cases}$$
 	If $P$ is the orthogonal projection of $L^2(\R)$ onto $PW_{\frac{1}{4}}$, then we can write
 	$$
 	Pf=\sum_{n\in \Z} \ip{Pf}{\psi_n}\phi_n=\sum_{n\in \Z} \ip{f}{\psi_n}\phi_n, \qquad \forall f\in L^2(\R)
 	$$
 	since $\psi_n\in PW_{\frac{1}{4}}$. In conclusion, $\{\phi_n\}_{n\in \Z}$ is a $K$-frame with $K=P$ (it fulfills \eqref{K-frame} as one can easily see by taking the Fourier transform)  but of course $\{\phi_n\}_{n\in \Z}$ is not contained in $\mathcal{R}(P)=PW_{\frac{1}{4}}$. Moreover, it is not even a \emph{frame sequence}, i.e. a frame for its closed span (indeed $\{\phi_n\}_{n\in \Z}$ does not satisfy \cite[Theorem 9.2.5]{ole}).

 In the literature there are many further studies or variations of $K$-frames  (see for example \cite{gj,guo,jf,nmg,NA,xzg} and the references therein).

In this paper we deal with two different generalizations of \cite{gavruta} which involve a {\em closed densely defined operator} $A$ on $\H$. When the operator is bounded, all definitions do coincide with those in \cite{gavruta}.   To justify our two different approaches, let us consider a Bessel sequence $\{g_n\}_{n\in\mathbb{N}}\subset\H$ and assume that, for $f\in \D(A)$, the domain of $A$, we have a decomposition
$$
Af=\sum_{n=1}^\infty a_n(f)g_n,
$$
for some $a_f:=\{a_n(f)\}_{n\in\mathbb{N}}\in \ell^2$; in particular, this situation appears when $\{g_n\}_{n\in\mathbb{N}}$ is a frame. If $A$ is unbounded, then the coefficients sequence $a_f$ \emph{can not depend continuously on $f$}, i.e. it can not exists $\gamma>0$ such that $\|a_f\|\leq\gamma\|f\|$ for every $f\in\mathcal{D}(A)$; this fact may represent another issue when we want to decompose  $\mathcal{R}(A)$ by a frame.

For these reasons, we develop two approaches where either the sequence $\{g_n\}_{n\in\mathbb{N}}$ or the coefficients sequence $a_f$ is what represents the unboundedness of $A$. To go into more details, in the first case we consider a non-Bessel sequence $\{g_n\}_{n\in\mathbb{N}}$ but the coefficients depend continuously on $f\in \D(A)$. In the second case, we take a Bessel sequence  $\{g_n\}_{n\in\mathbb{N}}$ and coefficients depending continuously on $f\in \D(A)$ only in the graph topology of $A$, which is stronger than the one of $\H$ when $A$ is unbounded.

The paper is organized as follows. After some preliminaries, see Section \ref{sec: preliminaries}, we introduce in Section \ref{sec:weak A-frame},  the notions of {\it weak $A$-frame} and {\it weak atomic system} for $A$   (Definitions \ref{def_weak A-Frames} and \ref{def: weak atomic system for A}, respectively),  where $A$ is a, possibly unbounded, densely defined operator. The word \emph{weak} is due to the fact that the decomposition of $\mathcal{R}(A)$, with $A$ also closable, holds only in a weak sense, in general; i.e., we find a Bessel sequence $\{t_n\}_{n\in\mathbb{N}}$ of $\H$ such that
\begin{equation*}
\ip{A f}{u}=\sum_{n=1}^\infty \langle f | t_n\rangle \ip{g_n}{u}\qquad \forall f\in \D(A), u\in \D(A^*)
\end{equation*} see Theorem \ref{th_char_weak_A_frame}. 
 Like in the bounded case (see \cite[Lemma 2.2]{NA}), we have also
\begin{equation*}
A^* u=\sum_{n=1}^\infty \ip{u}{g_n} t_n,\qquad \forall u\in \D(A^*),
\end{equation*}
and thus we note a change of the point of view: a weak $A$-frame may be used to get a strong decomposition of $A^*$ rather than $A$.

In Section \ref{sec:A-frame} we face our second approach, giving the general notions of {\it atomic system} for $A$ and {\it $A$-frame}, see Subsection \ref{subs: atomic system for A unbounded}, where $A$ is a, possibly unbounded, closed densely defined operator. Denote by  $\langle \cdot | \cdot\rangle_A $  the inner product which induces the graph norm $\|\cdot\|_A$ of $A$. The resulting decomposition is
$$
Af=\sum_{n=1}^\infty \langle f | k_n\rangle_A \; g_n\qquad \forall f\in \D(A),
$$
for some Bessel sequence $\{k_n\}_{n\in\mathbb{N}}$ of the Hilbert space $\D(A)[\|\cdot\|_A]$, see Corollary \ref{cor_A-frame_final}.
Actually,  this second approach is a particular case of $K$-frames, in the G\u{a}vru\c{t}a-like sense, where $K\in\B(\mathcal{J},\H)$ is a bounded operator between two different  Hilbert spaces $\mathcal{J}$ and $\H$, see Section \ref{sec:A-frame}. Indeed, for a densely defined closed operator $A$  on $\H$  we take $K=A$ and $\mathcal{J}=\D(A)[\|\cdot\|_A]$, see Corollary \ref{cor_A-frame_final}.

 Throughout the paper we give some examples of weak $A$-frames or  $A$-frames that can be obtained from  frames or that involve Gabor or wavelets systems.

\section{Preliminaries}\label{sec: preliminaries}

In the paper we consider an infinite dimensional Hilbert space $\H$  with inner product $\ip{\cdot}{\cdot}$ and norm $\|\cdot\|$.}
The term operator is used for a
linear mapping. Given an operator $F$, we denote its domain by $\mathcal{D}(F)$, its
range by $\mathcal{R}(F)$  and its adjoint by $F^*$, if $F$ is densely defined. By $\B(\H)$ we denote the set of bounded operators with domain $\H$ and we indicate by $\|F\|$ the usual norm of the operator $F\in\B(\H)$. In some examples we need the usual Hilbert spaces $L^2(0,1)$, $L^2(\mathbb{R})$  and the Sobolev spaces, denoted with standard notations, $H^1(0,1)$, $H_0^1(0,1)$, $H^1(\mathbb{R})$, see \cite[Section 1.3]{Schm}. 
As usual, we will indicate by $\ell^2$ the Hilbert space consisting of all sequences $x := \{x_n\}_{n\in\mathbb{N}}$ satisfying
$\sum_{n=1}^\infty | x_n |^ 2 < \infty$ with norm
$\|x\|_{2}=\left(\sum_{n=1}^\infty |x_{n}|^{2}\right)^{1/2}$. 

 We will say that a series $\sum_{n=1}^\infty g_n$, with $\{g_n\}_{n\in\mathbb{N}}\subset\H$,  is convergent to $g$ in $\H$ if $\lim_{n\to\infty}\|\sum_{k=1}^n g_k-g\|= 0$. We will write $\{g_n\}$ to mean a sequence $\{g_n\}_{n\in\mathbb{N}}$ of elements of $\H$. 
   For the following definitions the reader could refer e.g. to \cite{Anto_Bal,classif,ole95,ole,groc,heil}.

	A sequence $\{g_n\}$ of elements in $\H$ is a {\em Bessel sequence of $\H$}
	if any of the following equivalent conditions are satisfied, see \cite[Corollary 3.2.4]{ole}
	\begin{enumerate}
		\item[i)] there exists a constant $\beta > 0$ such that $\sum_{n=1}^\infty |\ip{f}{g_n}|^2 \leq \beta\|f\|^2$, for all $f\in \H$;
		\item[ii)] the series
		$\sum_{n=1}^\infty c_n g_n$ converges for all $c = \{c_n\}\in \ell^2$.
	\end{enumerate}

 A sequence $\{g_n\}$ of elements in $\H$ is a {\em lower semi-frame  for $\H$} with lower bound $\alpha > 0$ if
 $\alpha \|f\|^2\leq \sum_{n=1}^\infty |\ip{f}{g_n}|^2$, for every $f\in \H.$ Note that the series on the right hand side may diverge for some $f\in\H$.

A sequence $\{g_n\}$ of elements in $\H$ is a {\em frame for $\H$}
	if there exist $\alpha,\beta > 0$ such that $$\alpha \|f\|^2\leq \sum_{n=1}^\infty |\ip{f}{g_n}|^2 \leq \beta\|f\|^2, \qquad \forall f\in \H.$$

We now recall some operators which are classically used in the study of  sequences, see \cite{Anto_Bal,Anto_Bal2,classif,ole95}.
Let $\{g_n\}$ be a sequence of elements of $\H$. The {\it analysis operator} $C :\D(C )\subseteq \H\to \ell^2$ of $\{g_n\}$ is defined by $$\D(C )=\left \{f\in \H: \sum_{n=1}^\infty |\ip{f}{g_n}|^2<\infty\right  \}$$   $$C f=\{\ip{f}{g_n}\},\qquad \forall f\in \D(C).$$
The {\it synthesis operator} $D:\D(D)\subseteq \ell^2 \to \H$ of $\{g_n\}$ is defined on the dense domain
$$\D(D):=\left \{ \{c_n\}\in \ell^2:\sum_{n=1}^\infty c_n g_n \text{ is convergent in } \H\right \}$$  by
  $$D\{c_n\}=\sum_{n=1}^\infty c_n g_n,\qquad\forall\{c_n\}\in \D(D).$$
The {\it frame operator} $S:\D(S)\subseteq\H\to\H$ of $\{g_n\}$ is defined by
$$\D(S):=\left \{f\in \H: \sum_{n=1}^\infty \ip{f}{g_n}g_n \text{ is convergent in } \H\right \}$$  $$S f =\sum_{n=1}^\infty \ip{f}{g_n}g_n,\qquad \forall f\in \D(S).$$
The main properties of these operators are summarized below.

\begin{prop}[{\cite[Prop. 3.3]{classif}}]\label{prop: classif} Let $\{g_n\}$ be a sequence of $\H$. The following statements hold.
	\label{pro_oper_1}
	\begin{enumerate}
		\item[i)] $C=D^*$ and therefore $C$ is closed.
		\item[ii)] $D$ is closable if and only if $C$ is densely defined. In this case, $D \subseteq C^*$.
		\item[iii)] $D$ is closed if and only if $C$ is densely defined and $D=C^*$.
		\item[iv)] $S=D C$.
	\end{enumerate}
\end{prop}

A sequence $\{g_n\}$ is a Bessel sequence if and only if one, and then all, the operators $C,D$ and $S$ are bounded.  Moreover, if $\{g_n\}$ is a frame then $S$ is invertible with bounded inverse and the following reconstruction formula holds
\begin{equation}\label{rec_formula_frame}
f=\sum_{n=1}^\infty \ip{f}{h_n}g_n, \qquad f\in \H,
\end{equation}
where $\{h_n\}$ is a frame for $\H$ called a dual of  $\{g_n\}$. A  choice of $\{h_n\}$, which is always possible,  is $\{S^{-1}g_n\}$, called the {\em canonical dual} of $\{g_n\}$, but it can be different if $\{g_n\}$ is overcomplete, i.e. $\{g_n\}$ is not a basis. As a consequence of \eqref{rec_formula_frame}, the Hilbert space $\H$ must be separable.

Now we spend some words on non-Bessel sequences and reconstruction formulas. In general, if $\{g_n\}$ is a lower semi-frame, then by \cite[Proposition 3.4]{casazza} or \cite[Sect. 4]{corso2}, there exists a Bessel sequence $\{h_n\}$ such that
$$
h=\sum_{n=1}^\infty \ip{h}{g_n}h_n, \qquad\forall h\in \D(C).
$$
Hence a reconstruction formula holds in weak sense as
\begin{equation}
\label{weak_rec_1seq(a)}
\ip{f}{h}=\sum_{n=1}^\infty \ip{f}{h_n}\ip{g_n}{h}, \qquad f\in \H,h\in \D(C).
\end{equation}

Moreover, if $\D(C)$ is dense, then one can take $h_n=T^{-1}g_n$,
where $T:=|C|^2=C^*C$, a self-adjoint operator with bounded inverse on $\H$, see \cite{corso, corso2}.
The ``weakness'' of the formula \eqref{weak_rec_1seq(a)} is a consequence of the fact that the synthesis operator $D$ is not closed, in general.
If $\{g_n\}$ is a lower semi-frame, $\D(C)$ is dense and the synthesis operator $D$ of $\{g_n\}$ is closed, then  $D=C^*$, by Proposition \ref{pro_oper_1}. Thus $S=C^*C$ and the strong reconstruction formula again holds
\begin{equation*}
\label{strong_rec_lower}
f=S S^{-1}f=\sum_{n=1}^{\infty}\ip{f}{S^{-1}g_n}g_n, \qquad \forall f\in \H.
\end{equation*}

\berem
 In the light of \eqref{weak_rec_1seq(a)}, we compare the pair $(\{g_n\},\{h_n\})$ with reproducing pairs \cite{AST, Anto_Tp,  Speck_Bal_16, Speck_Bal_15}, weakly dual pairs \cite{LiOgawa}, also called pairs of pseudoframes for $\H$,   and pairs of pseudoframes for subspaces \cite{LiOgawa2}.  If in \eqref{weak_rec_1seq(a)} the formula holds for every $h\in \H$, then by definition $(\{g_n\},\{h_n\})$ is a weakly dual pair. In \eqref{weak_rec_1seq(a)}, if in addition $\D(C)$ is dense, the pair $(\{g_n\},\{h_n\})$ is a reproducing pair if and only if it is a weakly dual pair. In order the pair $(\{g_n\},\{h_n\})$ in \eqref{weak_rec_1seq(a)} to be a pseudoframe for $\D(C)$, this space has to be closed and  $\{g_n\}$ and $\{h_n\}$ have to be  Bessel sequences for $\D(C)$ and $\H$, respectively, so the nature of $\{g_n\}$ and $\{h_n\}$ in \eqref{weak_rec_1seq(a)} is very different from the setting of pseudoframe for subspace, in general.
\enrem

Now we recall the two notions we will generalize in the present paper. Let $K\in\B(\H)$. A sequence $\{g_n\}\subset\H$
is an {\em atomic system for $K$} \cite{gavruta} if the following statements hold\begin{itemize}\item[i)]  $\{g_n\}$ is a Bessel sequence of $\H$;
	\item[ii)] there exists $C > 0$ such that for every $f\in\H$ there exists $a_f = \{a_n(f)\}\in\ell^2$ such that $\|a_f\|\leq C\|f\|$ and $Kf =\sum_{n=1}^\infty a_n(f) g_n$.\end{itemize}

In \cite[Theorem 3]{gavruta}, the author proves the following
\begin{thm}
	\label{th_gavruta}
	Let $K\in \B(\H)$ and $\{g_n\}$ a sequence of $\H$. The following statements are equivalent.
	\begin{itemize}
		\item[i)] $\{g_n\}$ is an atomic system for $K$.
		\item[ii)] there exist constants $\alpha, \beta > 0$ such that
		$$\alpha\|K^*f\|^2\leq \sum_{n=1}^\infty |\ip{f}{g_n}|^2 \leq \beta\|f\|^2, \qquad \forall f\in \H.$$
		\item[iii)] there exists a Bessel sequence $\{h_n\}$ of $\H$ such that
		$$
		Kf=\sum_{n=1}^\infty \ip{f}{h_n}g_n, \qquad \forall  f\in \H.
		$$
	\end{itemize}
\end{thm}

Due to the inequalities in $ii)$ above, a sequence satisfying any of the conditions in Theorem \ref{th_gavruta} is also called a {\em $K$-frame for $\H$}.

Lastly, we will use the next lemma that can be obtained by Lemma 1.1 and Corollary 1.2 in \cite{BR}.

\begin{lemma}\label{lem: pseudoinverse unb} Let $\H$ and $\K$ be Hilbert spaces. Let $W:\D(W)\subset\K\to\H$ be a closed densely defined operator with closed range $\mathcal{R}(W)$. Then, there exists a unique  $W^\dag\in\B(\H,\K)$  such that $$\mathcal{N}(W^\dag)=\mathcal{R}(W)^\perp,\,\, \overline{\mathcal{R}(W^\dag)}=\mathcal{N}(W)^\perp,\,\, WW^\dag f=f, \qquad f\in\mathcal{R}(W).$$
\end{lemma}

The operator $W^\dag$ is called the {\em pseudo-inverse} of  $W$.

\section{Weak $A$-Frames and weak atomic systems for $A$}\label{sec:weak A-frame}

In this section we introduce our first generalization of the notion of $K$-frames to a densely defined operator on a Hilbert space $\H$.

\bedefi \label{def_weak A-Frames}
Let $A$ be a densely defined operator on $\H$. A {\it weak $A$-frame} for $\H$ is a sequence $\{g_n\}\subset\H$ such that
\begin{equation}\label{weak_A-Frames}
\alpha \|A^* f\|^2\leq\sum_{n=1}^\infty
|\ip{f}{g_n}|^2<\infty,
\end{equation}for every $f\in\D(A^*)$ and some $\alpha>0$.
\findefi

By \cite[Theorem 7.2]{heil}, if $A\in \B(\H)$ then $\{g_n\}$ is a weak $A$-frame if and only if it is an $A$-frame in the sense of \cite{gavruta}.

\berem\label{rem: uncond conv weak A frame} As it is clear from \eqref{weak_A-Frames}, the property of being a weak $A$-frame does not depend on the ordering of the sequence. \enrem

\berem  Let $A$ be a closable densely defined operator and $\{g_n\}$ a weak $A$-frame. The domain $\D(C)$ of the analysis operator $C$ of $\{g_n\}$ contains $\D(A^*)$. It is therefore dense and the synthesis operator $D$ is closable. Moreover,
\begin{equation*}
\alpha \|A^* f\|^2\leq\sum_{n=1}^\infty
|\ip{f}{g_n}|^2=\|C f\|^2=\|T^\frac{1}{2}f\|^2, \qquad\forall f\in \D(A^*),
\end{equation*}	
where $T=C^*C$. This shows that the series in \eqref{weak_A-Frames} is also bounded from above by the norm of a self-adjoint operator acting on $f\in \D(A^*)$. 
\enrem

\beex
Let $A$ be a densely defined operator on a separable Hilbert space $\H$. Then a weak $A$-frame for $\H$ always exists.
Indeed, let $\{e_n\}$ be an orthonormal basis for $\H$ contained in $\D(A)$ (there always exists such a one, by \cite[Ch. 1, Corollary 1]{young} and the Gram-Schmidt orthonormalization process),
it suffices to take $g_n=A e_n$, because for every $f\in\D(A^*)$,
$\|A^* f\|^2=\sum_{n=1}^\infty|\ip{f}{g_n}|^2$, by the Parseval identity.
\enex

\beex\label{ex: frame  to weak A frame}
Let $A$ be a densely defined operator on a separable Hilbert space $\H$. A more general example of weak $A$-frame is obtained by taking a frame $\{f_n\}\subset\D(A)$ for $\H$. In this case, in fact, there exist $\alpha,\beta>0$ such that
$$
\alpha\|A^*f\|^2\leq \sum_{n=1}^{\infty} |\ip{A^* f}{f_n}|^2\leq \beta \|A^*f\|^2, \qquad \forall f\in \D(A^*).
$$
Therefore, $\{A f_n\}$ is a weak $A$-frame for $\H$.
\enex

Now we generalize the notion of atomic system to the case of an unbounded operator.
\bedefi\label{def: weak atomic system for A} Let $A$ be a densely defined operator on $\H$. A {\it weak atomic system for $A$} is a sequence $\{g_n\}\subset \H$ such that
\begin{itemize}
	\item[i)] $\sum_{n=1}^\infty
	|\ip{f}{g_n}|^2<\infty$
	for every $f\in\D(A^*)$;
	\item[ii)] there exists $\gamma>0$ such that, for every $h\in\D(A)$, there exists $a_h=\{a_n(h)\}\in\ell^2$ satisfying $\|a_h\|\leq \gamma\|h\|$ and
\begin{equation}\label{eq: weak atomic system for $A$}
\ip{Ah}{u}=\sum_{n=1}^\infty a_n(h) \ip{g_n}{u}, \qquad \forall u\in \D(A^*).
\end{equation}
\end{itemize}
\findefi

\berem\label{unconditional convergence} If $\{g_n\}$ is a weak atomic system for $A$ then the series in \eqref{eq: weak atomic system for $A$} is  unconditionally  convergent. Indeed it is \emph{absolutely} convergent: fix any $h\in \D(A)$,  $u\in \D(A^*)$, then $\sum_{n=1}^\infty |a_n(h) \ip{g_n}{u}|\leq \|a_h\|\left( \sum_{n=1}^\infty |\ip{g_n}{u}|^2\right)^{1/2}< \infty$.
\enrem

The following lemma, which is a variation of \cite[Theorem 2]{doug},    will be useful in Theorem \ref{th_char_weak_A_frame} for a characterization of weak atomic systems for $A$ and weak $A$-frames.

\begin{lemma}
	\label{doug_var}
	Let $(\H,\|\cdot \|),(\H_1,\|\cdot \|_1)$ and $(\H_2,\|\cdot \|_2)$ be Hilbert spaces and $T_1:\D(T_1)\subseteq \H_1\to \H$, $T_2:\D(T_2)\subseteq \H\to \H_2$ densely defined operators. Denote by $T_1^*:\D(T_1^*)\subseteq \H\to \H_1$ and $T_2^*:\D(T_2^*)\subseteq \H_2\to \H$ the adjoint operators of $T_1,T_2$, respectively.
	Assume that
	\begin{enumerate}
		\item[i)] $T_1$ is closed;
		\item[ii)] $\D(T_1^*)=\D(T_2)$;
		\item[iii)] $\|T_1^* f\|_1\leq \lambda \|T_2f\|_2$ for all $f\in \D(T_1^*)$ and some $\lambda>0$.
	\end{enumerate}
	Then there exists an operator $U\in \B(\H_1,\H_2)$ such that $T_1=T_2^* U$.
\end{lemma}
\begin{proof}
	Define an operator $J$ on $R(T_2)\subseteq \H_2 $ as $JT_2f=T_1^* f\in \H_1$. Then $J$ is a well-defined bounded operator by $iii)$. Now we extend $J$ to the closure of $R(T_2)$ by continuity and define it to be zero on $R(T_2)^\perp$. Therefore $J\in \B(\H_2,\H_1)$ and $JT_2= T_1^* $, i.e. $T_1=T_2^* J^*$ and the statement is proved by taking $U=J^*$.	
\end{proof}

For the characterization in Theorem \ref{th_char_weak_A_frame} we need the following definition.
		\bedefi Let $A$ be a densely defined operator and $\{g_n\}$  a  sequence on $\H$, then a sequence $\{t_n\}$ of $\H$ is called a {\it weak $A$-dual of} $\{g_n\}$ if 	
	\begin{equation}
	\label{def_A_dual}
	\ip{A h}{u}=\sum_{n=1}^\infty \langle h | t_n\rangle \ip{g_n }{u}\qquad \forall h\in \D(A), u\in \D(A^*).
	\end{equation}
\findefi

\begin{thm}
	\label{th_char_weak_A_frame}
	Let $\{g_n\}\subset\H$ and $A$ a closable densely defined operator on $\H$. Then the following statements are equivalent.
	\begin{itemize}
		\item[i)] $\{g_n\}$ is a weak atomic system for $A$;
		\item[ii)] $\{g_n\}$ is a weak $A$-frame;
		\item[iii)] $\sum_{n=1}^\infty
		|\ip{f}{g_n}|^2<\infty$
		for every $f\in\D(A^*)$ and there exists a Bessel weak A-dual $\{t_n\}$.
	\end{itemize}
\end{thm}
\begin{proof} $i)\Rightarrow ii)$
	Let $f\in \D(A^*)$. Then
	$\|A^* f\|=\sup_{h\in \H,\|h\|=1}\left|\ip{A^*
			f}{h}\right|$ and by  the density of $\D(A)$ in $\H$ \begin{eqnarray*}\|A^* f\|&=&\sup_{h\in \D(A),\|h\|=1}\left|\ip{A^*
			f}{h}\right|=\sup_{h\in \D(A),\|h\|=1}|\ip{f}{Ah}|\\
		&=&\sup_{h\in \D(A),\|h\|=1}\left|\sum_{n=1}^\infty
		\overline{a_n(h)}\ip{f}{ g_n}\right|\\ &\leq&
		\sup_{h\in \D(A),\|h\|=1}\left(\sum_{n=1}^\infty
		|a_n(h)|^2\right)^{1/2}\left(\sum_{n=1}^\infty|\ip{f}{
			g_n}|^2\right)^{1/2}\\&\leq& \gamma_A\left(\sum_{n=1}^\infty|\ip{f}{
			g_n}|^2\right)^{1/2},\end{eqnarray*}
	taking into account that $\|a_h\|\leq \gamma_A\|h\|$ for some $\gamma_A>0$ and every $h\in \D(A)$. \\
	$ii)\Rightarrow iii)$ Let $\{e_n\}$ be an orthonormal basis of $\ell^2$. Consider the densely defined operator $B:\D(A^*)\to \ell^2$ given by $B f = \{\ip{f}{g_n}\}$ which is a restriction of the analysis operator $C:\D(C)\to \ell^2$.
	Since $C$ is closed, $B$ is closable. 
	
	We apply Lemma \ref{doug_var} to $T_1:=\overline{A}$ and $T_2:=B$ noting that $\|Bf\|^2=\sum_{n=1}^\infty
	|\ip{f}{g_n}|^2$. There exists $M\in \B(\H,\ell^2)$ such that $\overline{A}=B^*M$. 	This implies that for $h\in \D(A), u\in \D(A^*)=\D(B)$
	\begin{align*}
	\ip{A h}{u}&=\ip{B^* M h}{u}=\ip{M h}{B u }
	=\sum_{n=1}^\infty \ip{Mh}{e_n}\ip{g_n}{u}\\&=\sum_{n=1}^\infty \ip{h}{t_n}\ip{g_n}{u},
	\end{align*}
	taking $\{t_n\}=\{M^* e_n\}$ which is a Bessel sequence by \cite[Proposition 4.6]{classif}.	\\
	$iii)\Rightarrow i)$ It suffices to take $a_h=\{a_n(h)\}=\{\ip{h}{t_n}\}$ for all $h\in \D(A)$. Indeed for some $\gamma_A> 0$ we have $\sum_{n=1}^\infty |a_n(h)|^2= \sum_{n=1}^\infty|\ip{h}{t_n} |^2\leq \gamma_A \|h\|^2$ since $\{t_n\}$ is a Bessel sequence and $\ip{Ah}{u}=\sum_{n=1}^\infty a_n(h) \ip{g_n}{u}$, for $u\in \D(A^*)$.
\end{proof}

The term ``weak'' of weak $A$-frame and of weak atomic system, is due to the fact that \eqref{def_A_dual} holds whereas, in general, the same decomposition in strong sense $A h=\sum_{n=1}^\infty \langle h | t_n\rangle g_n$ may fail, unlike the case of $A$-frame where $A\in \B(\H)$, see \cite[Theorem 3]{gavruta}. We show this with the following example.
\beex
Suppose that $\H$ is separable. Let $\{e_n\}$ be an orthonormal basis for $\H$ and $\{g_n\}$ the sequence defined by $g_1=e_1$ and $g_n=n(e_n-e_{n-1})$ for $n\geq 2$. We denote by $C,D$ the analysis and synthesis operators of $\{g_n\}$, respectively. As it is shown in \cite{ole95}, $C$ is densely defined and $D$ is a proper restriction of $C^*$. In particular, $\left\{\frac{1}{n}\right\}_{n\in \N}\in \D(C^*)\backslash \D(D)$. Let $\mathcal{I}$ be the analysis operator of $\{e_n\}$. Obviously it is a bijection in $\B(\H,\ell^2)$. Now consider the sesquilinear form
$$
\Omega(f,u)=\sum_{n=1}^\infty \ip{f}{e_n}\ip{g_n}{u},
$$
which is defined on $\H\times \D(C)$. Moreover $\Omega (f,u)=\ip{\mathcal{I} f}{C u }$ for all $f\in \H,u\in \D(C)$. Therefore $\Omega (f,u)=\ip{C^*\mathcal{I} f}{u}$ for all $f\in \D(C^*\mathcal{I}),u\in \D(C)$.

This suggests to define $A:=C^*\mathcal{I}$ which is a densely defined closed operator. The adjoint $A^*$ is equal to $\mathcal{I}^*C$ and then it has $\D(C)$ as domain. Thus
$$
\ip{A f}{u}=\sum_{n=1}^\infty \ip{f}{e_n}\ip{g_n}{u}, \qquad \forall f\in \D(A),u\in \D(A^*),
$$
i.e. $\{g_n\}$ is a weak $A$-frame by Theorem \ref{th_char_weak_A_frame}. But the relation
$$
Af=\sum_{n=1}^\infty \ip{f}{e_n}g_n, \qquad\forall f\in \D(A)
$$
does not hold. Indeed, the element $f:=\sum_{n=1}^\infty\frac{1}{n}e_n$ belongs to $\D(A)$ and  the sum $\sum_{k=1}^n \ip{f}{e_k}g_k=e_n$ for $n\in \N$, does not converge in $\H$.
\enex

\beex\label{exm_A_duals}
In general, for a weak $A$-frame $\{g_n\}$ for $\H$ a Bessel weak $A$-dual $\{t_n\}$ is not unique. For all examples we have considered we give here a possible choice of $\{t_n\}$.
\begin{enumerate}
	\item[i)] If $\{g_n\}:=\{Ae_n\}$, where $\{e_n\}\subset\D(A)$ is an orthonormal basis for $\H$, then one can take $\{t_n\}=\{e_n\}$.
	\item[ii)] If $\{g_n\}:=\{Af_n\}$, where $\{f_n\}\subset\D(A)$ is a frame for $\H$, then one can take for $\{t_n\}$ any dual frame of $\{f_n\}$.
\end{enumerate}
\enex

		\berem \label{rem_adj_weak}
		Let $A$ be a densely defined operator, $\{g_n\}$  a weak $A$-frame and $\{t_n\}$  a Bessel weak $A$-dual of $\{g_n\}$, then for $ h\in\D(A)$ and $ u\in \D(A^*)$ \begin{equation*}
		\ip{A^* u}{h}=\ip{u}{Ah}=\sum_{n=1}^\infty \overline{\langle h | t_n\rangle \ip{g_n}{u}}= \sum_{n=1}^\infty \ip{u}{g_n} \langle t_n| h\rangle.
		\end{equation*}
		Since the sequence $\{t_n\}$ is Bessel, the series $\sum_{n=1}^\infty \ip{u}{g_n} t_n$ is convergent. Therefore
		$$
		\ip{A^* u}{h}=
		\left\langle {\sum_{n=1}^\infty \ip{u}{g_n} t_n}\Big{|}  {h}\right\rangle,\qquad \forall h\in \D(A), u\in \D(A^*)
		$$
		and by the density of $\D(A)$ we obtain
		\begin{equation}
		\label{A^*_dual}
		A^* u=\sum_{n=1}^\infty \ip{u}{g_n} t_n,\qquad \forall u\in \D(A^*).
		\end{equation}
		
		In conclusion, it is worth noting that in this setting, surprisingly,  from condition \eqref{weak_A-Frames} the {\em strong} decomposition of $A^*$ follows, whereas for $A$ we have just a {\em weak} decomposition, in general. If $A$ is symmetric, i.e. $A\subset A^*$, then clearly from \eqref{A^*_dual} we have a decomposition of $A$ in strong sense. If $\{g_n\}$ is also a Bessel sequence, then $A$ is bounded on its domain, thus closable, and  condition \eqref{weak_A-Frames} gives us decompositions in strong sense for both the closure $\overline{A}$ and $A^*$  (see  \cite[Theorem 3]{gavruta} and \cite[Lemma 2.2]{NA}).
		\enrem

\berem One could ask whether a weak $A$-dual $\{t_n\}$ of a weak $A$-frame $\{g_n\}$ is a weak $A^*$-frame, with $A$ a closable  densely defined operator. The answer is negative, in general. Indeed, if $\{t_n\}$ is a Bessel sequence, an inequality as
$$
\alpha\|Af\|^2\leq \sum_{n=1}^{\infty} |\ip{f}{t_n}|^2, \qquad \forall f\in \D(A)
$$
with $\alpha>0$, implies that $A$ is bounded on its domain. \enrem

Under further assumption of $A$, weak $A$-frames can be used to decompose the domain of $A^*$. 	
	\begin{thm}\label{thm: interchang unb}
			Let $A$ be a densely defined closed operator with $\mathcal{R}(A)=\H$ and $(A^\dag)^*\in\B(\H)$ the adjoint of the pseudo-inverse $A^\dag$ of $A$. Let $\{g_n\}$ be a weak $A$-frame and $\{t_n\} $ a Bessel weak $A$-dual of $\{g_n\}$. Then, the sequence $\{h_n\}$, with $h_n:=(A^\dag)^* t_n\in\H$ for every  $n\in\mathbb{N}$, is Bessel and \begin{equation*}
			u=\sum_{n=1}^\infty
			\ip{u}{g_n} h_n,\qquad u \in\mathcal{D}(A^*).\end{equation*}
	\end{thm}
	\begin{proof}		
		First observe that, since $A$ is onto,  $f = AA^\dag f$, for every $f \in\H$. Let  $\{g_n\}$, $\{t_n\}$ and $\{h_n\}$ be as in the statement.  Then, by \eqref{def_A_dual},  we have that for $f\in\H,u\in \D(A^*)$
	\begin{align*}
	\ip{f}{u}=\ip{AA^\dag f}{u}=\sum_{n=1}^\infty \langle {A^\dag f | t_n\rangle} \ip{g_n}{u}=\sum_{n=1}^\infty \langle { f |h_n\rangle} \ip{g_n}{u}
	\end{align*}
	and for some $\gamma>0$
		\begin{eqnarray*}\sum_{n=1}^\infty |\ip{f}{h_n}|^2=
			\sum_{n=1}^\infty |\ip{A^\dag f}{t_n} |^2\leq\gamma\| A^\dag f\|^2\leq\gamma \|A^\dag \|^2\|f\|^2\end{eqnarray*} since $\{t_n\}$ is Bessel for $\H$ and $A^\dag$ is bounded.
		Hence, $\{h_n\}$ is a  Bessel sequence of $\H$.
	 Finally, for any $f\in\H$, $u\in\mathcal{D}(A^*)$, we have
				$\ip{u}{f} =\sum_{n=1}^\infty
		\ip{\ip{u}{g_n}h_n}{f}$. Since the sequence $\{h_n\}$ is Bessel, the series  $
		\sum_{n=1}^\infty
		\ip{u}{g_n} h_n$ is convergent and we conclude that
		$
		u=\sum_{n=1}^\infty
		\ip{u}{g_n} h_n,$ for all $ u \in\mathcal{D}(A^*). $
\end{proof}

Now we give another theorem of characterization for weak $A$-frames involving the synthesis operator.

\begin{thm}
	Let $A$ be a closed densely defined operator, $\{g_n\}\subset \H$ and $D:\D(D)\subset\ell^2 \to \H$ the synthesis operator of $\{g_n\}$.
	The following statements are equivalent.
	\begin{enumerate}
		\item[i)] The sequence $\{g_n\}$ is  a weak $A$-frame for $\H$;
		\item[ii)] there exists a densely defined, closed extension $R$ of $D$ such that $A=RQ$ with some $Q\in \B(\H,\ell^2)$;
		\item[iii)] there exists a closed densely defined operator $L:\D(L)\subset\ell^2\to\H$ such that and $\D(A^*)\subset\D(L^*)$, $g_n=L e'_n$ where $\{e'_n\}\subset\D(L)$ is an orthonormal basis for $\ell^2$ and $A=LU$ for some $U\in \B(\H,\ell^2)$.
	\end{enumerate}
\end{thm}
\begin{proof}
	$i)\Rightarrow ii)$ Following the proof of Theorem \ref{th_char_weak_A_frame}, $A=B^*M$. Then the statement is proved taking $Q=M$ and $R=B^*$, since $B^* \supseteq C^*\supseteq D$.\\
	$ii)\Rightarrow iii)$ Since $R$ is an extension of the syntesis operator $D$,  it suffices to take $L=R,U=M$ and  $\{e'_n\}$ the canonical orthonormal basis of $\ell^2$.  \\
	$iii)\Rightarrow i)$
	For every $f\in\D(A^*)$ the adjoint of $L$ is given by
	$$L^* f=\sum_{n=1}^\infty\ip{f}{g_n}e'_n.$$ Indeed, for $ c\in\ell^2$
	\begin{eqnarray*}
		\ip{L^* f}{c } &=& \ip{L^* f}{\sum_{n=1}^\infty c_n e'_n } = \sum_{n=1}^\infty \overline{c_n}\ip{f}{ L e'_n}\\
		&=&  \sum_{n=1}^\infty\ip{ e'_n}{c }\ip{f}{ g_n}= \ip{\sum_{n=1}^\infty\ip{f}{g_n} e'_n}{c }.
	\end{eqnarray*}
Moreover, $\{g_n\}$ is a weak $A$-frame because for every $f\in \D(A^*)$ we have $\sum_{n=1}^\infty
|\ip{f}{g_n}|^2=\|L^* f\|^2<\infty$ and 
	$\|A^*f\|^2\leq\|U^*L^* f\|^2\leq\|U\|^2\|L^* f\|^2.
	$
\end{proof}

We conclude this section with some concrete examples.

\beex
\label{exm1}
Let us consider the differential operator $Af=-if'$ with domain $H^1(0,1)$ which is a densely defined closed operator on $\H=L^2(0,1)$, see \cite[Section 1.3]{Schm}.  The sequence $\{g_n\}_{n\in \Z}=\{e_{nb}\}_{n\in \Z}$, where  $0<b\leq1$ and $e_{nb}(x)=e^{2\pi i n b x}$ for $x\in (0,1)$, is a frame for $L^2(0,1)$, see \cite[Section 9.8]{ole}. Therefore $\{Ag_n\}=\{2\pi n b e_{bn}\}$ is a weak $A$-frame for $L^2(0,1)$ by Example \ref{ex: frame  to weak A frame}. The canonical dual frame of $\{e_{nb}\}$ is $\{\frac{1}{b}e_{nb}\}$, then according to Example \ref{exm_A_duals} we can take $\{\frac{1}{b}e_{nb}\}$ as weak $A$-dual of $\{g_n\}$.  The adjoint $A^*$ is the operator $A^*f=-if'$ with $\D(A^*)=H^1_0(0,1)$, see again \cite[Section 1.3]{Schm}. Note that $A^*\subset A$.
Hence the decomposition in weak sense of Theorem \ref{th_char_weak_A_frame} reads as
$$
\ip{-if'}{h}=\ip{A f}{h}=\sum_{n\in \Z} 2\pi n \ip{f}{e_{nb}}\ip{e_{nb}}{h}, \qquad \forall f\in H^1(0,1), h\in H^1_0(0,1).
$$
Finally, we have also a strong decomposition of $A^*$ by \eqref{A^*_dual}:
$$
-if'=A^* f=\sum_{n\in \Z} 2\pi n\ip{f}{e_{nb}}e_{nb}, \qquad \forall f\in H^1_0(0,1).
$$
\enex

\beex
\label{exm2}
Let $\H:=L^2(\mathbb{R})$ and denote by $A$ the selfadjoint operator $Af=-if'$ with domain $\D(A)=H^1(\mathbb{R})$.
Let $g:\mathbb{R}\to \mathbb{C}$ be a continuous and differentiable function with support $[0,L]$, more generally, one can take a function $g\in H^1(\mathbb{R})$ such that $g\in W$ where $W$ is the Wiener space, see e.g. \cite[Section 11.5]{ole} for the definition of $W$. 

Let $y\in \mathbb{R}$, $\omega\in \mathbb{R}$ and $T_y,M_\omega:\H\to \H$ be the {\em translation} and {\em modulation} operators defined, for $f\in \H$,  by $(T_yf)(x)=f(x-y)$ and $(M_\omega f)(x)=e^{2\pi i\omega x}f(x)$, respectively. Consider the Gabor system  $G(g,a,b)$. By the hypothesis, $\{g_{m,n}\}_{m,n\in\mathbb{Z}}\subseteq \D(A)$. Assume in particular that $\{g_{m,n}\}_{m,n\in\mathbb{Z}}$ is a frame for $L^2(\mathbb{R})$, a necessary and sufficient condition is given in \cite[Theorem 6.4.1]{groc}. Then, by Example \ref{ex: frame  to weak A frame}, $\{Ag_{m,n}\}_{m,n\in\mathbb{Z}}$ is a weak $A$-frame; i.e., for some $\gamma>0$
$$\gamma\|A^*f\|^2\leq \sum_{m,n\in \mathbb{Z}} |\ip{f}{Ag_{m,n}}|^2 <\infty \qquad \forall f\in \D(A^*)=\D(A)=H^1(\mathbb{R}).$$

Explicitly,
\begin{align*}
Ag_{m,n}(x)&=2\pi  b n e^{2\pi ib n x}g(x-a m)-ie^{2\pi ib n x}g'(x-a m)\\
&= 2\pi  b n (M_{b n}T_{a m}g)(x) -i(M_{bn}T_{a m}g')(x).
\end{align*}

For the decomposition of $A$ we can use the canonical dual of the Gabor frame $\{g_{m,n}\}_{m,n\in\mathbb{Z}}$ which is a Gabor frame $\{h_{m,n}\}_{m,n\in\mathbb{Z}}$ with some window $h\in L^2(0,1)$. Since $A$ is selfadjoint we can write directly a decomposition in strong sense of $A$ according to \eqref{A^*_dual}
$$
-if'=A f=\sum_{m,n\in \Z} \ip{f}{M_{bn}T_{a m}(2\pi b n g-ig')} M_{bn}T_{a m} h, \qquad \forall f\in H^1(\R).
$$

 Once more we  point out  that the property of being a weak $A$-frame does not depend on the ordering of the sequence $\{M_{bn}T_{a m}(2\pi b n g-ig')\}_{m,n\in\mathbb{Z}}$, see Remark \ref{rem: uncond conv weak A frame}.\enex

\beex
Let us consider the same space $\H:=L^2(\mathbb{R})$ and the operator $Af=f'$ with domain $\D(A)=H^1(\mathbb{R})$. Let $\phi\in H^1(\mathbb{R})$ and the {\em shift-invariant system} $\{\phi_k(x)\}_{k\in \Z}:=\{\phi(x-ck)\}_{k\in \Z}$, with $c>0$. Then $\{(A\phi_k)(x)\}_{k\in \Z}=\{\phi'(x-ck)\}_{k\in \Z}$. However, we cannot apply Example \ref{ex: frame  to weak A frame} to say that $\{A\phi_k\}$ is a weak $A$-frame. Indeed, as it is known \cite{ole}, $\{\phi_k\}$ is never a frame for $L^2(\mathbb{R})$.

Consider instead the {\em wavelet system} $\{\phi_{m,n}\}_{m,n\in \Z}:=\{a^{-\frac{m}{2}}\phi(a^{-m}x-nb)\}_{m,n\in \Z}$ with $a,b>0$. We have $\{\phi_{m,n}\}_{m,n\in\mathbb{Z}}\subset H^1(\mathbb{R})$ and $$\{(A\phi_{m,n})(x)\}_{m,n\in\mathbb{Z}}=\{a^{-\frac{3m}{2}}\phi'(a^{-m}x-nb)\}_{m,n\in\mathbb{Z}}.$$ The sequence we obtained is nothing but the  wavelet system $\{\phi'_{m,n}\}_{m,n\in\mathbb{Z}}$ generated by the derivative $\phi'$ multiplied by the scalars $\{a^{-m}\}_{m\in\mathbb{Z}}$. 

When $\{\phi_{m,n}\}_{m,n\in\mathbb{Z}}$ is a frame for $\H$, $\{A\phi_{m,n}\}_{m,n\in\mathbb{Z}}$ is a weak $A$-frame. In particular, by \cite[Theorem 10.6 (c)]{groc}, for any $k \in \N$, there exists a function $\phi$ with compact support  and continuous derivatives up to order $k$ such that $\{\phi_{m,n}\}_{m,n\in \Z}:=\{2^{-\frac{m}{2}}\phi(2^{-m}x-n)\}_{m,n\in \Z}$ is an orthonormal basis for $L^2(\mathbb{R})$ and hence  $\{A\phi_{m,n}\}_{m,n\in\mathbb{Z}}$  is a weak $A$-frame.
\enex

\beex
Let $A$ be a closed and densely defined on $\H$. The domain $\D(A)$ of $A$ can be turned into a Hilbert space if endowed with the graph norm $\|\cdot \|_A$. Denote it by $\H_A$ and by $\H_A^\times$  its conjugate dual and construct the rigged Hilbert space  $\H_A \hookrightarrow\H\hookrightarrow\H_A^\times$,  where $\hookrightarrow$ means that  the embeddings $\H_A \subset\H\subset\H_A^\times$ are continuous with dense range, see e.g. \cite[Chapter 10]{ait_book}. Since the sesquilinear form $B(\cdot,\cdot)$ that puts  $\H_A$ and $\H_A^\times$ in
duality is an extension of the inner product of $\H$ we write  $B(\xi,f)=\ip{\xi}{f}$  for the action of $\xi\in \H_A^\times$ on $f\in \H_A$. 

Now let $\{g_n\}\subset \H$. Then $\{g_n\}$ can be regarded as a sequence in $\H_A^\times$. Assume that it is a {\it Bessel-like sequence} in the sense of \cite[Definition 2.10]{beltp}, i.e. for every bounded
subset $\cM\subset\H_A$,
\begin{equation*} \sup_{f\in \cM}\sum_{n=1}^\infty|\ip{f}{g_n}|^2<\infty.\end{equation*} Then, by \cite[Proposition 2.11]{beltp}, $\sum_{n=1}^\infty |\ip{f}{g_n}|^2<\infty $ and the operator $F:\H_A \to \ell^2$ given by $Ff:=\{\ip{f}{g_n}\}$ is bounded. If $F$ is also injective, e.g. if $\{g_n\}$ is dense in $\H$, and has closed range, then $\{g_n\}$ is a weak $A^*$-frame since
$$
c\|Af\|^2\leq c\|f\|_A^2\leq\sum_{n=1}^\infty |\ip{f}{g_n}|^2 =\|Ff\|^2<\infty,\qquad \forall f\in \D(A)
$$
and for some $c>0$.
\enex

\section{Atomic systems for bounded operators \\ between different Hilbert spaces}\label{sec:A-frame}

In this section we will give another generalization of the notions and results in \cite{gavruta} to unbounded closed densely defined operators in a Hilbert space. { If $A:\D(A)\to\H$ is a closed and densely defined operator, then it can be seen as a bounded operator $A:\H_A\to \H$ between two different Hilbert spaces, where by $\H_A$ we indicate the Hilbert space $\D(A)[\|\cdot\|_A]$ with $\|\cdot\|_A$ the graph norm. 
	
Thus, before going forth, we reproduce the main definitions and results in \cite{gavruta} for bounded operators from a  Hilbert space $\J$ into another, say $\H$, omitting the proofs since they are very similar to the standard ones where $\J=\H$, \cite{gavruta,NA}. We will come back to the operator $A:\H_A\to \H$ in Section  \ref{subs: atomic system for A unbounded}.}
\\

Let $\ip{\cdot}{\cdot}_\H,\ip{\cdot}{\cdot}_\J$ be the inner products and $\|\cdot\|_\H,\|\cdot\|_\J$ the norms of $\H$ and $\J$, respectively. We denote by $\B(\J,\H)$ the set of bounded linear operators from $\J$ into $\H$.

\bedefi Let $K\in \B(\J,\H)$. An {\it atomic system for $K$} is a sequence $\{g_n\}\subset \H$ such that
\begin{itemize}
\item[(i)]  $\{g_n\}$ is a Bessel sequence,
\item[(ii)] there exists $\gamma>0$ such that for all $f\in\J$ there exists $a_f=\{a_n(f)\}\in\ell^2$, with
  $\|a_f\|\leq \gamma\|f\|_\J$ and $Kf=\sum_{n=1}^\infty a_n(f) g_n$.
\end{itemize}\findefi

Clearly the previous notion reduces to that of atomic system in \cite{gavruta} when $\J=\H$.

\beex\label{exist thm}Let $\H$ be separable and $K\in \B(\J,\H)$. 
Every frame $\{g_n\}$ for $\H$ is an atomic system for $K$. Indeed, if $\{v_n\}$ is a dual frame of  $\{g_n\}$, then 
\begin{equation*}
Kf=\sum_{n=1}^\infty \ip{Kf}{v_n}_\H g_n, \qquad\forall f\in \J
\end{equation*}
and the definition is satisfied by taking $a_f=\{\ip{Kf}{v_n}_\H\}$ for $f\in \J$.
\enex

\beex\label{ex: atomic syst} Let  $\J$ be separable, $K\in \B(\J,\H)$ and $\{f_n\}$ a frame for $\J$ with dual frame $\{h_n\}\subset\J$, then for all $f\in \J$
$$f=\sum_{n=1}^\infty \ip{f}{h_n}_\J f_n, \mbox{ hence } Kf=\sum_{n=1}^\infty \ip{f}{h_n}_\J Kf_n.$$
Thus the sequence $\{g_n\}=\{Kf_n\}$ is
 an atomic system for $K$, taking $a_f=\{a_n(f)\}:=\{\ip{f}{h_n}_\J\}$.

\enex

For $L\in \B(\J,\H)$  we denote by $L^*\in \B(\H,\J)$ its adjoint.
We now give a characterization of the atomic
systems for operators in $\B(\J,\H)$ similar to that obtained by  G\u{a}vru\c{t}a in
\cite[Theorem 3]{gavruta}.

\begin{thm}
  \label{th_char_K_frame}
  Let $\{g_n\}\subset\H$ and  $K\in \B(\J,\H)$. Then the following are equivalent.
  \begin{itemize}
    \item[i)] $\{g_n\}$ is an atomic system for $K$;
    \item[ii)] there exist $\alpha, \beta>0$  such that for every $f\in \H$
    \begin{equation}\label{double ineq}
     \alpha \|K^* f\|_\J^2\leq\sum_{n=1}^\infty
|\ip{f}{g_n}_\H|^2\leq \beta \|f\|_\H^2;
    \end{equation}
    \item[iii)] $\{g_n\}$ is a Bessel sequence of $\H$ and there exists a  Bessel sequence $\{k_n\}$ of $\J$  such that
        \begin{equation}\label{eq: two Bess seq}
          Kf=\sum_{n=1}^\infty \langle f | k_n\rangle_\J\, g_n,\qquad \forall f\in \J.
        \end{equation}
  \end{itemize}
\end{thm}

\bedefi \label{def_K_frame} Let $K\in \B(\J, \H)$. A sequence $\{g_n\}\subset\H$ is called a \emph{$K$-frame for $\H$} if the chain of inequalities \eqref{double ineq} holds true for all $f\in\H$ and some $\alpha, \beta>0$.  \findefi

By \eqref{eq: two Bess seq} the range $\mathcal{R}(K)$ must be a separable subspace of $\H$, which may be non separable.
As in \cite[Definition 2.1]{NA} a sequence $\{k_n\}\subset \J$ as in \eqref{eq: two Bess seq} is called a {\it $K$-dual} of the $K$-frame $\{g_n\}\subset\H$.

\beex As in Section \ref{sec:weak A-frame}, we remark that, in general, a $K$-dual $\{k_n\}\subset\J$ of a $K$-frame $\{g_n\}\subset\H$  is not unique. Then, for the $K$-frames $\{g_n\}$ considered in Examples \ref{exist thm} and \ref{ex: atomic syst} we give possible $K$-duals.
\begin{enumerate}
	\item[i)] If $\{g_n\}:=\{f_n\}$, with $\{f_n\}\subset\H$ a frame for $\H$, then one can take $\{k_n\}=\{K^* v_n\}$ where $\{v_n\}$ is any dual frame of $\{f_n\}$.
	\item[ii)] If $\{g_n\}:=\{Kf'_n\}$, with $\{f'_n\}\subset\J$ a frame for $\J$, then one can take for $\{k_n\}$ any dual frame of $\{f'_n\}$.
\end{enumerate}
\enex

Once at hand a $K$-frame $\{g_n\}$, the Bessel sequence  $\{k_n\}\subset\J$ in Theorem \ref{th_char_K_frame} is a $K^*$-frame, see \cite[Lemma 2.2]{NA} for the case $\J=\H$.

We now give a characterization of $K$-frames involving the synthesis operator. The equivalence of the first two sentences is an easy generalization of \cite[Theorem 4]{gavruta} and the other ones are straightforward.

\begin{thm}	\label{char_K_frame}
	Let $K\in \B(\J,\H)$, $\{g_n\}\subset \H$ and $D:\D(D)\subseteq \ell^2\to\H$ the synthesis operator of $\{g_n\}$. The following statements are equivalent.
	\begin{enumerate}
		\item[i)] $\{g_n\}$ is a $K$-frame for $\H$;
		\item[ii)] there exists $L\in \B(\ell^2,\H)$ such that $g_n=L e'_n$ where $\{e'_n\}$ is an orthonormal basis for $\ell^2$ and  $\mathcal{R}(K)\subset\mathcal{R}(L)$;
		\item[iii)] $D\in \B(\ell^2,\H)$ and $\mathcal{R}(K)\subset\mathcal{R}(D)$;
		\item[iv)] $D\in \B(\ell^2,\H)$ and there exists $M\in \B(\J,\ell^2)$ such that $K=DM$.
	\end{enumerate}
\end{thm}

From Theorem \ref{char_K_frame} $iii)$ it follows that a $K$-frame is not necessarily a frame sequence, indeed the range of the synthesis operator may be not closed, see \cite[Corollary 5.5.2]{ole}.

\subsection{Atomic systems for unbounded operators $A$ and $A$-frames}\label{subs: atomic system for A unbounded} As announced at the beginning of this section, we come back to our original aim to generalize $K$-frames, with $K\in\B(\H)$, in the context of unbounded closed and densely defined operator $A$ on a Hilbert space $\H$. Here, for simplicity, we denote again by $\ip{\cdot}{\cdot}$ and $\|\cdot\|$ the inner product and the norm of $\H$, respectively. \\

From now on we will consider $A$ as a bounded operator in $\B(\H_A,\H)$, where $\H_A$ is the Hilbert space obtained endowing the domain $\D(A)$ with the graph norm $\|\cdot\|_A$, induced by the graph inner product $\ip{\cdot}{\cdot}_A$. Let $A^\sharp:\H\to \H_A$ be the adjoint of  $A:\H_A\to\H$, different from $A^*$ the adjoint of  the unbounded operator $A$.\\

For the reader's convenience we rewrite the definitions of atomic system for $A\in\B(\H_A,\H)$ and of $A$-frame. A sequence $\{g_n\}\subset\H$ is said to be

   \begin{enumerate}
   	\item[i)] an {\it atomic system} for $A$ if $\{g_n\}$ is a Bessel sequence and there exists $\gamma>0$ such that for all $f\in\D(A)$ there exists $a_f=\{a_n(f)\}\in\ell^2$, with
   	$\|a_f\|\leq \gamma\|f\|_A$ and $Af=\sum_{n=1}^\infty a_n(f) g_n$, with respect to the norm of $\H$;
   	\item[ii)] an {\it $A$-frame} if there exist $\alpha, \beta>0$  such that for every $f\in \H$
   	\begin{equation*}
   	\alpha \|A^\sharp f\|_A^2\leq\sum_{n=1}^\infty
   	|\ip{f}{g_n}|^2\leq \beta \|f\|^2.
   	\end{equation*}
     \end{enumerate}
Hence, Theorem \ref{char_K_frame} can be rewritten as follows.

\begin{cor}
	\label{cor_A-frame_final}
	Let $\{g_n\}\subset\H$ and  $A$  a closed densely defined operator on $\H$. Then the following are equivalent.
	\begin{itemize}
		\item[i)] $\{g_n\}$ is an atomic system for $A$;
		\item[ii)] $\{g_n\}$ is an $A$-frame;
		\item[iii)] $\{g_n\}$ is a Bessel sequence of $\H$ and there exists a Bessel sequence $\{k_n\}$ of $\H_A$  such that
		\begin{equation}\label{exp_A}
		Af=\sum_{n=1}^\infty \langle f | k_n\rangle_A \; g_n,\qquad \forall f\in \D(A)
		\end{equation} with respect to the norm of $\H$.
		\item[iv)] the synthesis operator $D$ of $\{g_n\}$ is bounded and everywhere defined on $\ell^2$ and $\mathcal{R}(A)\subset\mathcal{R}(D)$;
		\item[v)] the synthesis operator $D$ of $\{g_n\}$ is bounded and everywhere defined on $\ell^2$ and there exists $M\in \B(\H_A,\ell^2)$ such that $A=DM$.
	\end{itemize}
\end{cor}

Note also that if $A\in \B(\H)$, then the graph norm of $A$ is defined on $\H$ and it is equivalent to $\|\cdot\|$, thus our notion  reduces to that of \cite{gavruta}.

\berem
The expansion in \eqref{exp_A} of $Af$ in terms of $\{g_n\}$ involves the inner product $\ip{\cdot}{\cdot}_A$. One might ask if there exists also a sequence $\{t_n\}\subset \H$ such that
$$
Af=\sum_{n=1}^\infty \ip{f}{t_n}g_n,\qquad \forall f\in \D(A)
$$
like for atomic systems for  $A\in \B(\H)$, see \cite[Theorem 3]{gavruta}. The answer, in general, is negative if $A$ is unbounded.
Indeed, let $\{e_n\}$ be an orthonormal basis for a separable Hilbert space $\H$ and $A$ an unbounded closed and densely defined operator in $\H$. Assume in particular  that  $\{e_n\}\nsubseteq \D(A^*)$, such an orthonormal basis for $\H$ can always  be  found. Clearly, $\{e_n\}$ is an $A$-frame.  Suppose that there exists a sequence $\{t_n\}\subset \H$ such that
$Af=\sum_{n=1}^\infty \ip{f}{t_n}e_n$, for all $f\in \D(A).$
Then $\ip{Af}{e_n}=\ip{f}{t_n}$ for all $f\in \D(A)$ and $n\in \mathbb{N}$. But this leads to the contradiction that $\{e_n\}\subset \D(A^*)$.
\enrem

We conclude by showing an example of an $A$-frame which is not a frame.
\beex
\label{exm_not_frame}
Let $\H=L^2(\R)$, $\{\alpha_k\}_{k\in \Z}$ be a complex sequence and $A$ the closed and densely defined operator on $L^2(\R)$ defined as
$$
(Af)(x)=\begin{cases}
\alpha_{k}f(x) \qquad\qquad\!\! x\in [2k,2k+1[\\
\alpha_{k}f(x-1) \qquad x\in [2k+1,2k+2[
\end{cases}
$$
 where $k$ varies in $\Z$, with natural domain
$$\D(A)=\left \{f\in L^2(\R): \sum_{k\in \Z} |\alpha_k|^2\int_{2k}^{2k+1} |f(x)|^2dx<\infty \right\}.$$  The operator $A\in \B(L^2(\R))$ if and only if $\{\alpha_k\}_{k\in \Z}$ is bounded.

Now let $g\in L^2(\R)$ be bounded with support $[0,2]$ and let the essential infimum of $|g|$ on $[0,2]$ be positive,  $\displaystyle \text{essinf}_{x\in [0,2]} |g(x)|>0$. Consider the Gabor system $\mathcal{G}(g,a,b):=\mathcal{G}(g,2,1)=\{e^{2\pi i m x}g(x-2n)\}_{m,n\in \Z}$; it is Bessel because $g$ is bounded and compactly supported, but it is not a frame since $ab=2>1$. However, we show that it is an $A$-frame. Indeed, the range of the synthesis operator of $\mathcal{G}(g,1,2)$ is
$$
\mathcal{R}(D)=\{f\in L^2(\R): f(x)=f(x-1), \forall x\in [2k+1,2k+2[, \forall k\in\mathbb{Z}\}
$$
and contains $\mathcal{R}(A)$. Therefore, by Corollary \ref{cor_A-frame_final}, $\mathcal{G}(g,2,1)$ is an $A$-frame.
\enex

\section{Conclusions}
{ In conclusion, we make some remarks to highlight the novelty and potential applications of the notion of weak $A$-frame. If $\{f_n\}\subset\H$ is a frame for $\H$ and $\{h_n\}\subset\H$ is a dual frame of $\{f_n\}$, then a closable densely defined operator $A$ in $\H$ can be decomposed as follows: $$Af=\sum_{n=1}^\infty \ip{Af}{h_n}f_n,\qquad\forall f\in \D(A).$$
	
However, in this decomposition the action of the operator $A$ still appears. On the contrary, if $\{g_n\}\subset\H$ is a weak $A$-frame, then by Theorem \ref{th_char_weak_A_frame} there exists a Bessel sequence $\{t_n\}\subset\H$ such that
	\begin{equation*}
	\ip{A h}{u}=\sum_{n=1}^\infty \langle h | t_n\rangle \ip{g_n }{u},\qquad \forall h\in \D(A), u\in \D(A^*)
	\end{equation*}
and the action of the operator $A$ does not appear in the decomposition.
Since we have also
	\begin{equation*}
	A^* u=\sum_{n=1}^\infty \ip{u}{g_n} t_n,\qquad \forall u\in \D(A^*)
	\end{equation*}
weak $A$-frames are clearly connected to multipliers that have been recently object of many studies, refer e.g. to the survey \cite{Balazs_surv}. However, few works were directed to unbounded multipliers, so our study could give a contribution in this direction, actually it is what we did in Examples \ref{exm1} and \ref{exm2} for some specific operators. 

We want to mention 	\cite{BagBell,Bag_Riesz,Bag_sesq,inoue} where some unbounded multipliers have been defined as model of non-selfadjoint Hamiltonians. Let us focus on \cite{Bag_Riesz} for a connection with weak $A$-frames. Fixed a complex sequence $\alpha=\{\alpha_n\}$ and a Riesz basis $\phi=\{\phi_n\}$ with dual $\psi=\{\psi_n\}$, one can construct the operator
\begin{equation}
\label{Hamiltonian}
H^{\alpha}_{\phi,\psi}f=\sum_{n=1}^\infty \alpha_n \ip{f}{\psi_n}\phi_n
\end{equation}
with $\D(H^{\alpha}_{\phi,\psi})$ being the greatest subspace where \eqref{Hamiltonian} converges. Then $\{\alpha_n\phi_n\}$ is a weak $H^{\alpha}_{\phi,\psi}$-frame, indeed by \cite[Proposition 2.1]{Bag_Riesz} 
$$\D({H^{\alpha}_{\phi,\psi}}^*)=\left \{f\in \H: \sum_{n=1}^\infty  |\ip{f}{\alpha_n\phi_n}|^2<\infty\right\}$$  and thus Theorem \ref{th_char_weak_A_frame} iii) is satisfied.
}

\section*{Acknowledgements}
The authors warmly thank Prof. C. Trapani and the referees for their  fruitful comments and remarks.  This work has been supported by the
Gruppo Nazionale per l'Analisi Matematica, la Probabilit\`{a} e le
loro Applicazioni (GNAMPA) of the Istituto Nazionale di Alta
Matematica (INdAM).

\vspace*{0.15cm}

\bibliographystyle{amsplain}

\end{document}